\documentclass[a4paper]{article}
\usepackage{latexsym}
\usepackage{amsmath}
\usepackage{amssymb}
\usepackage{amsthm}
\usepackage{pstricks}
\usepackage{pst-node,pst-tree,pst-eps,xspace}
\IfFileExists{pst-beta.tex}
             {\input{pst-beta.tex}}
             {\RequirePackage{pst-tree}}

\usepackage{graphicx}
\usepackage{latexsym}
\newcommand{\bsb}[1]{\boldsymbol{#1}}

\newtheorem{The}{Theorem}
\newtheorem{Lem}[The]{Lemma}

\newtheorem{Pro}{Proposition}
\theoremstyle{definition}
\newtheorem{De}[The]{Definition}
\newtheorem{Exa}{Example}
\newtheorem{Rem}{Remark}

\def\maj{{\scriptstyle \mathsf{MAJ}}}

\def\SS{\frak S}

\def\llsq{[\hspace{-0.3mm}[}
\def\rrsq{]\hspace{-0.3mm}]}
\title{Generating permutations with a given major index}

\author{
Vincent {\sc Vajnovszki}\\ 
{\small LE2I, Universit\'e de Bourgogne}\\
{\small BP 47870, 21078 Dijon Cedex, France}\\
{\small \tt vvajnov@u-bourgogne.fr}
}

\begin{document}
\maketitle

\large

\begin{abstract}

In [S. Effler, F. Ruskey, A CAT algorithm for listing permutations with 
a given number of inversions, {\it I.P.L.}, 86/2 (2003)]
the authors give an algorithm, which appears to be CAT, for 
generating permutations with a given major index.
In the present paper we give a new algorithm for generating 
a Gray code for subexcedant sequences. We show that this
algorithm is CAT and
derive it into a CAT generating algorithm for 
a Gray code for permutations
with a given major index.
\end{abstract}

\section{Introduction}

We present the first guaranteed constant average time generating algorithm for 
permutations with a fixed index. 
First we give a co-lex order generating algorithm
for bounded compositions. Changing its generating order and 
specializing it for particular classes of compositions
we derive a generating algorithms for 
a Gray code for fixed weight subexcedant sequences;
and after some improvements we obtain an efficient 
version of this last algorithm.
The generated Gray code has the remarkable property that
two consecutive sequences differ in at most
three adjacent positions and by a bounded amount in these positions.
Finally applying a bijection introduced in \cite{Vaj_11} between subexcedant sequences and 
permutations with a given index we derive the desired algorithm,
where consecutive generated permutations differ by at most three 
transpositions.

Often, Gray code generating algorithms can be re-expressed simpler
as algorithms with the same time complexity and generating the same class of objects, 
but in different ({\em e.g.} lexicographical) order. 
This is not the case in our construction: the {\em Grayness} of the 
generated subexcedant sequences  is critical in the construction of
the efficient algorithm generating permutations with a fixed
index.

\medskip

A {\em statistic} on the set $\SS_n$ of length $n$ permutations
is an association of an element of $\mathbb{N}$
to each permutation in $\SS_n$.
For $\pi\in\SS_n$ the {\em  major index}, $\maj$, is a statistic
defined by (see, for example, \cite[Section 10.6]{Lothaire_83})

$$
\displaystyle \maj\, \pi = \mathop{\sum_{1\leq i <n}}_{\pi_i>\pi_{i+1}} i.
$$

\begin{De}
For two integers $n$ and $k$, an {\it $n$-composition of $k$}
is an $n$-sequence $\bsb{c}=c_1c_2\ldots c_n$ of non-negative integers
with $\sum_{i=1}^n c_i=k$.
For an $n$-sequence $\bsb{b}=b_1b_2\ldots b_n$, $\bsb{c}$ is said 
{\it $\bsb{b}$-bounded if  $0\leq c_i\leq b_i$}, for all $i$, $1\leq i\leq n$.
\end{De}

In this context $b_1b_2\ldots b_n$  is called  {\it bounding sequence}
and we will consider only bounding sequences
with either $b_i>0$ or $b_i=b_{i-1}=\ldots =b_1=0$
for all $i$, $1\leq i\leq n$. Clearly, $b_i=0$ is equivalent
to fix $c_i=0$.
We denote by $C(k,n)$ the set of all $n$-compositions of $k$,
and by $C^{\bsb{b}}(k,n)$ the set of $\bsb{b}$-bounded $n$-compositions of $k$;
and if $b_i\geq k$ for all $i$, then $C^{\bsb{b}}(k,n)=C(k,n)$.

\begin{De}
A {\it subexcedant sequence} $\bsb{c}=c_1c_2\ldots c_n$ is an $n$-sequence with 
$0\leq c_i\leq i-1$, for all $i$; and $\sum_{i=1}^nc_i$ is called
the {\it weight} of $\bsb{c}$.
\end{De}

We denote by $S(k,n)$ the set 
of length $n$ and weight $k$ subexcedant sequences,
and clearly $S(k,n)=C^{\bsb{b}}(k,n)$ with $\bsb{b}=0\,1\,2\,\ldots \,(n-1)$.

\section{Generating fixed weight subexcedant sequences}


We give three  generating algorithms, and the third one generates 
efficiently combinatorial objects in bijection with 
permutations having fixed index :
\begin{itemize}
\item {\tt Gen\_Colex} generates the set $C^{\bsb{b}}(k,n)$
      of bounded compositions in co-lex order (defined later). 
\item {\tt Gen1\_Gray} which is obtained from  {\tt Gen\_Colex} by: 
      \begin{itemize}
      \item changing its generating order, and
      \item restricting it to the bounding sequence $\bsb{b}=01\ldots (n-1)$. 
      \end{itemize}
      It produces a Gray code for the set $S(k,n)$,
      and it can be seen as the definition of this Gray code.
\item {\tt Gen2\_Gray} is a an efficient version of {\tt Gen1\_Gray}. 
\end{itemize}

Finally, in Section~\ref{gen_perms}, regarding the subexcedant
sequences in $S(k,n)$ as McMahon permutation codes
(defined in Section~\ref{sec_Mc_code}), a constant average time 
generating algorithm for a Gray code for the set of permutations of length $n$
with the major index equals $k$ is obtained.

\subsection{Algorithm {\tt Gen\_Colex}}

This algorithm generates
$C^{\bsb{b}}(k,n)$ in {\it co-lex order}, which is defined as:
$c_1c_2\ldots c_n$ precedes $d_1d_2\ldots d_n$ in co-lex order
if $c_nc_{n-1}\ldots c_1$ precedes $d_nd_{n-1}\ldots d_1$ in lexicographical
order. Its worst case time complexity is  $O(k)$ per composition.

For a set of bounded compositions $C^{\bsb{b}}(k,n)$,
an {\it increasable position} (with respect to $C^{\bsb{b}}(k,n)$) in 
a sequence $c_1c_2\ldots c_n\notin C^{\bsb{b}}(k,n)$
is an index $i$ such that:
\begin{itemize}
\item $c_1=c_2=\ldots c_{i-1}=0$, and
\item there is a composition $d_1d_2\ldots d_n\in C^{\bsb{b}}(k,n)$ 
with $c_i<d_i$ and $c_{i+1}=d_{i+1}$, $c_{i+2}=d_{i+2}$, \dots, 
$c_n=d_n$.
\end{itemize}

For example, for $C^{01233}(3,5)$ the increasable positions
are underlined in the following sequences: $0\,0\,\underline{0}\,\underline{1}\,0$ 
and $0\,\underline{0}\,2\,0\,0$. 
Indeed, the first two positions in $0\,0\,0\,1\,0$ are not
increasable since there 
is no composition in $C^{01233}(3,5)$ with the suffix $0\,1\,0$;
and the third position in $0\,0\,2\,0\,0$ is not increasable
because $2$ is the maximal value
in this position.
Clearly, if $\ell<r$ are two increasable positions in $\bsb{c}$, 
then each $i$, $\ell<i<r$, is still an increasable position in $\bsb{c}$
(unless $b_i=0$).

Here is the sketch of the co-lex order
generating procedure for $C^{\bsb{b}}(k,n)$:

\begin{itemize}
\item[$\bullet$] initialize $\bsb{c}$ by the length $n$ sequence $0\,0\,\ldots\, 0$;
\item[$\bullet$] for each increasable position $i$ in $\bsb{c}$, increase
      $c_i$ by one and call recursively the generating
      procedure if the obtained sequence $\bsb{c}$ is not a composition in $C^{\bsb{b}}(k,n)$,
      and output it elsewhere.
\end{itemize}


The complexity of the obtained algorithm is $O(k)$ per generated composition
and so inefficient. Indeed, too many nodes in the generating tree induced by 
this algorithm have degree one. Algorithm {\tt Gen\_Colex} in Figure 
\ref{algo_colex} avoids some of these nodes.
We will identify a node in a generating tree by the corresponding
value of the sequence $\bsb{c}$; and
a {\it redundant node} in a generating tree 
induced by the previous sketched algorithm is a node with 
a unique successor and which differs in the same position
from its ancestor and its successor.
For example, in Figure \ref{two_tree} (a) redundant nodes are:
$0\,0\,0\,1$, $0\,0\,0\,2$, $0\,0\,1\,3$, $0\,0\,2\,3$ and $0\,1\,3\,3$.
These nodes occur when, for a given suffix, the 
smallest value allowed in an increasable position in the
current sequence $\bsb{c}$ is not $1$, and this position is
necessarily $\ell$, the leftmost increasable one. Algorithm {\tt Gen\_Colex} 
avoids redundant nodes by setting $c_{\ell}$ to its minimal value $e=k-\sum_{j=1}^{\ell-1}b_j$
(and $\sum_{j=1}^{i}b_j$ can be computed for each $i$, $1\leq i\leq n$,
in a pre-processing step). 
For example, in Figure \ref{two_tree} (b) there are no redundant nodes. 
However, in the generating tree induced by {\tt Gen\_Colex}
there still remain arbitrary length sequences of successive nodes  
with a unique successor; they are avoided in procedure {\tt Gen2\_Gray}.
 
Algorithm  {\tt Gen\_Colex} is given in Figure \ref{algo_colex} 
where $\ell$ is the leftmost increasable position in the 
current sequence $\bsb{c}$, and $r$ the leftmost non-zero position in $\bsb{c}$,
and thus the rightmost increasable 
position in $\bsb{c}$ is $r$ if $c_r<b_r$ and $r-1$ elsewhere ($b_1b_2\ldots b_n$ being the
bounding sequence).
The main call is {\tt Gen\_Colex($k$,$n$)}
and initially $\bsb{c}$ is $0\,0\,\ldots\, 0$.
(As previously, in this algorithm the function 
$k\mapsto \min\{s\,|\,\sum_{j=1}^s b_j\geq k\}$
can be computed and stored in an array, in a pre-processing step.)

The induced generating tree for the call {\tt Gen\_Colex($4$,$5$)} is given 
in Figure \ref{fig_2_trees} (a).

\begin{figure}
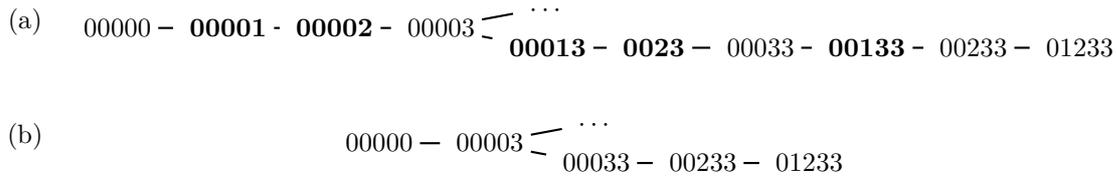

\begin{tabular}{cc}
(a) & {
\psset{levelsep=14mm,treesep=3mm,radius=-0.5mm}
\pstree[nodesep=1mm,treemode=R]{\TR{$00000$}}
{ \pstree{\TR{  $\mathbf{00001}$}}
  { \pstree{\TR{ $\mathbf{00002}$}}
    { \pstree{\TR{ $00003$}}
      { \pstree{\TR{ $\cdots$}}{}
        \pstree{\TR{ $\mathbf{00013}$}}
	{ \pstree{\TR{ $\mathbf{0023}$}}
	  { \pstree{\TR{ $00033$}}
	    { \pstree{\TR{ $\mathbf{00133}$}}
	      { \pstree{\TR{ $00233$}}
	         {\pstree{\TR{ $01233$}}{}
		 }
	      }	    
	    }
	  }
	}
      }
    }
  }  
}
}\\
\\
\\
(b) & 
{
\psset{levelsep=14mm,treesep=3mm,radius=-0.5mm}
\pstree[nodesep=1mm,treemode=R]{\TR{$00000$}}
{ \pstree{\TR{  $00003$}}
      {  \pstree{\TR{ $\cdots$}}{}
         \pstree{\TR{ $00033$}}
	    {  \pstree{\TR{ $00233$}}
	         {\pstree{\TR{ $01233$}}{}
		 }
	      }	    
	    }
}
}
\end{tabular}
\caption{
\label{two_tree}
The path from the root $0\,0\,0\,0\,0$ to the composition $0\,1\,2\,3\,3\in C^{01234}(9,5)$:
(a) before deleting redundant nodes (in boldface);
and (b) in the generating tree 
induced by the call of {\tt Gen\_Colex($9,5$)} where redundant nodes are avoided. 
}
\end{figure}

\begin{figure}[h]
\begin{center}
\begin{tabular}{|c|}
\hline
{\tt
\begin{minipage}[c]{.46\linewidth}
\begin{tabbing}
\hspace{0.75cm}\=\hspace{1.cm}\=\hspace{0.5cm}\=\hspace{1.6cm}\=\hspace{1.9cm}
            \=\hspace{1.5cm}\kill
procedure Gen\_Colex($k$,$r$)\\
\>global $n,c,b$;\\
\> if $k=0$\\
\> then print $c$;\\
\> else \> if $c[r]=b[r]$ \\
\>      \>then $r:=r-1$; \\
\>      \>end if\\
\>      \> $\ell:=\min\{s\,|\,\sum_{j=1}^s b[j]\geq k\}$;\\
\>      \> for $i:=\ell$ to $r$ do\\        
\>      \> \> if $i=\ell$ \> then $e:=k-\sum_{j=1}^{\ell-1}b[j]$;\\
\>      \> \>             \> else  $e:=1$;\\
\>      \> \> end if\\
\>      \> \> $c[i]:=c[i]+e$;\\
\>      \> \> Gen\_Colex($k-e$,$i$);\\   
\>      \> \> $c[i]:=c[i]-e$;\\
\>      \>end do\\
\> end if\\
end procedure.
\end{tabbing}
\end{minipage}}\\
\hline
\end{tabular}
\caption{\label{algo_colex} Algorithm {\tt Gen\_Colex}.}
\end{center}
\end{figure}

\begin{figure}
\begin{tabular}{cc}
  \begin{tabular}{|c|}
  \hline
  \\
{
\small
\psset{levelsep=15mm,treesep=3mm,radius=-0.5mm}
\pstree[nodesep=1mm,treemode=R]{\TR{$00000$}}
{ \pstree{\TR{  $00010$}}
  { \pstree{\TR{ $00210$}}
    { \pstree{\TR{ $\bsb{01210}$}}{}
    }
    \pstree{\TR{ $00020$}}
    {\pstree{\TR{ $00120$}}
     {\pstree{\TR{ $\bsb{01120}$}}{}
      \pstree{\TR{ $\bsb{00220}$}}{}
     }
     \pstree{\TR{ $00030$}}
     {\pstree{\TR{ $\bsb{01030}$}}{}
      \pstree{\TR{ $\bsb{00130}$}}{}
     }
    }
  }
 \pstree{\TR{ $00001$}}
 { \pstree{\TR{ $00201$}}
   { \pstree{\TR{ $\bsb{01201}$}}{}
   }
   \pstree{\TR{ $00011$}}
   {\pstree{\TR{ $00111$}}
    {\pstree{\TR{ $\bsb{01111}$}}{}
     \pstree{\TR{ $\bsb{00211}$}}{}
    }
    \pstree{\TR{ $00021$}}
    {\pstree{\TR{ $\bsb{01021}$}}{}
     \pstree{\TR{ $\bsb{00121}$}}{}
     \pstree{\TR{ $\bsb{00031}$}}{}
    }
   }
   \pstree{\TR{ $00002$}}
   {\pstree{\TR{ $00102$}}
    {\pstree{\TR{ $\bsb{01102}$}}{}
     \pstree{\TR{ $\bsb{00202}$}}{}
    }
    \pstree{\TR{ $00012$}}
    {\pstree{\TR{ $\bsb{01012}$}}{}
     \pstree{\TR{ $\bsb{00112}$}}{}
     \pstree{\TR{ $\bsb{00022}$}}{}
    }
    \pstree{\TR{ $00003$}}
    {\pstree{\TR{ $\bsb{01003}$}}{}
     \pstree{\TR{ $\bsb{00103}$}}{}
     \pstree{\TR{ $\bsb{00013}$}}{}
     \pstree{\TR{ $\bsb{00004}$}}{}
    }
   }
 }  
}
} 
  \\
  \\
  \hline
  \end{tabular}
& 
  \begin{tabular}{|c|}
  \hline
  \\ 
{
\small
\psset{levelsep=15mm,treesep=3mm,radius=-0.5mm}
\pstree[nodesep=1mm,treemode=R]{\TR{$00000$}}
{ \pstree{\TR{  $00010$}}
  { \pstree{\TR{ $00210$}}
    { \pstree{\TR{ $\bsb{01210}$}}{}
    }
    \pstree{\TR{ $00020$}}
    {\pstree{\TR{ $00030$}}
     {\pstree{\TR{ $\bsb{01030}$}}{}
      \pstree{\TR{ $\bsb{00130}$}}{}
     }
     \pstree{\TR{ $00120$}}
     {\pstree{\TR{ $\bsb{00220}$}}{}
      \pstree{\TR{ $\bsb{01120}$}}{}
     }
    }
  }
 \pstree{\TR{ $00001$}}
 { \pstree{\TR{ $00002$}}
   {\pstree{\TR{ $00102$}}
    {\pstree{\TR{ $\bsb{01102}$}}{}
     \pstree{\TR{ $\bsb{00202}$}}{}
    }
    \pstree{\TR{ $00012$}}
    {\pstree{\TR{ $\bsb{01012}$}}{}
     \pstree{\TR{ $\bsb{00112}$}}{}
     \pstree{\TR{ $\bsb{00022}$}}{}
    }
    \pstree{\TR{ $000003$}}
    {\pstree{\TR{ $\bsb{00004}$}}{}
     \pstree{\TR{ $\bsb{00013}$}}{}
     \pstree{\TR{ $\bsb{00103}$}}{}
     \pstree{\TR{ $\bsb{01003}$}}{}
    }
   }
   \pstree{\TR{ $00011$}}
   {\pstree{\TR{ $00021$}}
    {\pstree{\TR{ $\bsb{01021}$}}{}
     \pstree{\TR{ $\bsb{00121}$}}{}
     \pstree{\TR{ $\bsb{0003}1$}}{}
    }
    \pstree{\TR{ $00111$}}
    {\pstree{\TR{ $\bsb{00211}$}}{}
     \pstree{\TR{ $\bsb{01111}$}}{}
    }
   }
   \pstree{\TR{ $00201$}}
   {\pstree{\TR{ $\bsb{01201}$}}{}
   }
 }  
}
}
  \\
  \\
  \hline
  \end{tabular}
\\
(a) & (b)
\end{tabular}
\caption{\label{fig_2_trees}
(a): The tree induced by the call of {\tt Gen\_Colex($4$,$5$)}
with $\bsb{b}=0\,1\,2\,3\,4$, and (b): that induced
by {\tt Gen1\_Gray($4$,$5$)}. Terminal nodes are in bold-face }
\end{figure}

\subsection{Algorithm {\tt Gen1\_Gray}}  

This algorithm is defined in Figure \ref{algos_Gen1_Gray}
and is derived from {\tt Gen\_Colex}:  the order of recursive calls
is changed according to a direction (parameter $dir$), and it is specialized for 
bounding sequences $\bsb{b}=0\,1\,2\,\ldots\, (n-1)$, and so
it produces subexcedant sequences. 
It has the same time complexity as {\tt Gen\_Colex} and 
we will show that it produces a Gray code.

The call of {\tt Gen1\_Gray} with $dir=0$ produces, in order,
a recursive call with $dir=0$, then $r-\ell$ calls in the {\tt for} statement
with $dir$ equals successively:
\begin{itemize}
\item $0,1,\ldots 0,1$, if $r-\ell$ is even, and 
\item $1,0,\ldots 1,0,1$, if $r-\ell$ is odd.
\end{itemize}
In any case, the value of $dir$ corresponding to the last call is $1$.

The call of {\tt Gen1\_Gray} with $dir=1$ produces the same
operations as previously but in reverse order,
and in each recursive call the value of $dir$ is replaced by $1-dir$.
Thus, the call of {\tt Gen1\_Gray} with $dir=1$ produces,
in order, $r-\ell$ calls in the {\tt for} statement
with $dir$ equals alternatively $0,1,0,\ldots$, then a last call with $dir=1$.
See Figure \ref{fig_2_trees} (b) for an example of generating tree 
induced by this procedure.


Let $\mathcal{S}(k,n)$ be the {\it ordered list}
for $S(k,n)$ generated by the call {\tt Gen1\_Gray($k$,$n$,$0$)}, and
it is easy to see that $\mathcal{S}(k,n)$ is suffix partitioned, 
that is, sequences with the same suffix are contiguous;
and Theorem \ref{main_th} shows that $\mathcal{S}(k,n)$ is a Gray code.

For a sequence $\bsb{c}$, a $k\geq 1$ and $dir\in \{0,1\}$
we denote by  $\mathrm{first}(k;dir;\bsb{c})$ and $\mathrm{last}(k;dir;\bsb{c})$, 
the first and last subexcedant sequence produced by the
call of {\tt Gen1\_Gray$(k,r,dir)$} if the current 
sequence is $\bsb{c}$, and $r$ the position of the leftmost non-zero value 
in $\bsb{c}$.
In particular, if $\bsb{c}=0\,0\,\ldots\,0$, then
$\mathrm{first}(k;0;\bsb{c})$ is the first sequence in $\mathcal{S}(k,n)$,
and $\mathrm{last}(k;0;\bsb{c})$ the last one.

\newpage
\begin{Rem}$ $
\label{rem_2_points}
\begin{enumerate}
\label{rev_01_rem_reverse}
\item For a sequence $\bsb{c}$, the list produced by the call
      {\tt Gen1\_Gray$(k,r,0)$} is the reverse of the list 
      produced by the call {\tt Gen1\_Gray$(k,r,1)$}, and with the 
      previous notations we have
      \begin{eqnarray*}
      \mathrm{last}(k;dir;\bsb{c})=\mathrm{first}(k;1-dir;\bsb{c}),
      \end{eqnarray*}
      for $dir\in\{0,1\}$.
\item Since the bounding sequence is $\bsb{b}=0\,1\,\ldots\, (n-1)$
      it follows that, for $\bsb{c}=0\,0\,\ldots\, 0\,c_ic_{i+1}\ldots c_n$, 
      $c_i\neq 0$, $\mathrm{first}(k;0;\bsb{c})$ is
      \begin{itemize}
      \item $a_1a_2\ldots a_{i-1}c_ic_{i+1}\ldots c_n$ if 
      $k\leq\sum_{j=1}^{i-1}(j-1)=\frac{(i-1)\cdot(i-2)}{2}$,
      where $a_1a_2\ldots a_{i-1}$ is the smallest sequence, in co-lex order, in
      $S(k,i-1)$,
      \item $a_1a_2\ldots a_ic_{i+1}\ldots c_n$ if $k>\frac{(i-1)\cdot(i-2)}{2}$,
      where $a_1a_2\ldots a_i$ is the smallest sequence, in co-lex order, in
      $S(k+c_i,i)$.
\end{itemize}
\end{enumerate}
\end{Rem}

\begin{figure}[h]
\begin{tabular}{|c|}
\hline
\begin{minipage}[c]{.46\linewidth}
{\tt
\begin{tabbing}\hspace{1cm}\=\
\hspace{1.cm}\=\hspace{1.2cm}\=\hspace{2.cm}\=\hspace{1.9cm}
            \=\hspace{1.5cm}\kill
procedure Gen1\_Gray($k$,$r$,$dir$)\\
global $n,c,b$;\\
if $k=0$\\
then output $c$;\\
else \> if $c[r]=r-1$ \\
     \>then $r:=r-1$; \\
     \>end if\\     
     \> $\ell:=\min\{s\,|\,\frac{s(s-1)}{2} \geq k\}$;\\
     \> $e:=k-\frac{(\ell-1)(\ell-2)}{2}$;\\
     \> if $dir=0$ \\
     \> then \> $c[\ell]:=c[\ell]+e$; Gen1\_Gray($k-e$,$\ell$,$0$); $c[\ell]:=c[\ell]-e$;\\
     \>      \> $dir:=(r-\ell) \mod 2$;\\
     \>      \> for $i:=\ell+1$ to $r$ do\\
     \>      \>  \>$c[i]:=c[i]+1$; Gen1\_Gray($k-1$,$i$,$dir$); $dir:=(dir+1)\mod 2$; $c[i]:=c[i]-1$;\\
     \>      \>  end do\\
     \> else \> $dir:=0$;\\
     \>      \> for $i:=r$ downto $\ell+1$ do\\
     \>      \>  \>$c[i]:=c[i]+1$; Gen1\_Gray($k-1$,$i$,$dir$); $dir:=(dir+1)\mod 2$; $c[i]:=c[i]-1$;\\
     \>      \>  end do\\
     \>      \> $c[\ell]:=c[\ell]+e$; Gen1\_Gray($k-e$,$\ell$,$1$); $c[\ell]:=c[\ell]-e$;\\
     \> end if\\
end if\\
end procedure.
\end{tabbing}
}
\end{minipage}\\
\hline
\end{tabular}
\caption{\label{algos_Gen1_Gray} Algorithm {\tt Gen1\_Gray}, the Gray code 
counterpart of {\tt Gen\_Colex} specialized to subexcedant sequences.}
\end{figure}

Now we introduce the notion of close sequences.
Roughly speaking, two sequences are close if they differ in at most
three adjacent positions and by a bounded amount in these positions.
Definition \ref{3_tuple} below defines formally this notion, and 
Theorem \ref{main_th} shows that consecutive subexcedant
sequences generated by {\tt Gen1\_Gray} 
are close. 


Let $\bsb{s}=s_1s_2\ldots s_n$ and $\bsb{t}=t_1t_2\ldots t_n$ be two subexcedant sequences
of same weight which differ in at most three adjacent positions, and let $p$ 
be the rightmost of them (notice that necessarily $p\geq 3$).
The {\it difference} between $\bsb{s}$ and $\bsb{t}$
is the $3$-tuple 
$$
(a_1,a_2,a_3)=(s_{p-2}-t_{p-2},s_{p-1}-t_{p-1},s_p-t_p).
$$
Since $\bsb{s}$ and $\bsb{t}$ have same weight
it follows that $a_1+a_2+a_3=0$; and we denote by $-(a_1,a_2,a_3)$
the tuple $(-a_1,-a_2,-a_3)$.

\begin{De}
\label{3_tuple}
Two sequences $\bsb{s}$ and $\bsb{t}$ in $S(k,n)$ are 
{\it close} if:
\begin{itemize}
\item $\bsb{s}$ and $\bsb{t}$ differ in at most three adjacent positions, and
\item if $(a_1,a_2,a_3)$ is the difference between
$\bsb{s}$ and $\bsb{t}$, then 
$$
(a_1,a_2,a_3)\in \{\pm(0,1,-1),\pm(0,2,-2),\pm(1,-2,1),\pm(1,-3,2),\pm(1,1,-2),\pm(1,0,-1)\}.
$$
\end{itemize}
\end{De}

Even if the second point of this definition sound somewhat
arbitrary, it turns out that consecutive sequences generated by 
algorithm {\tt Gen1\_Gray} are close under this definition, 
and our generating algorithm for permutations with a given 
index in Section \ref{gen_perms} is based on it.

\begin{Exa}
The following sequences are close:
$0\underline{12}01$ and $0\underline{03}01$; 
$010\underline{03}$ and $010\underline{21}$;
$0\underline{020}1$ and $0\underline{101}1$;
$01\underline{132}$ and $01\underline{204}$;
the positions where the sequences differ are underlined.
Whereas the following sequences are not close:
$0\underline{0211}$ and $0\underline{1030}$ (they differ in more than
three positions); 
$01\underline{201}$ and $01\underline{030}$ (the difference $3$-tuple
is not a specified one).
\end{Exa}

\begin{Rem}
\label{rem_inter}
If $\bsb{s}$ and  $\bsb{t}$ are two close 
subexcedant sequences in $S(k,n)$, then there are at most two `intermediate' 
subexcedant sequences $\bsb{s'}$, $\bsb{s''}$ in $S(k,n)$
such that the differences 
between $\bsb{s}$ and $\bsb{s'}$,
between $\bsb{s'}$ and $\bsb{s''}$, and
$\bsb{s''}$ and $\bsb{t}$ are 
$\pm(1,-1,0)$.
\end{Rem}

\begin{Exa}
\label{un_example}
Let $\bsb{s}=0\,1\,0\,1\,1\,1$ and $\bsb{t}=0\,0\,2\,0\,1\,1$
be two sequences in $S(4,6)$. Then $\bsb{s}$ and $\bsb{t}$ are close since they 
difference is $(1,-2,1)$, and there is one `intermediate' sequence 
$\bsb{s'}=0\,0\,1\,1\,1\,1$ in $S(4,6)$ with
\begin{itemize}          
\item the difference between $\bsb{s}$ and $\bsb{s'}$ is $(1,-1,0)$,
\item the difference between $\bsb{s'}$ and $\bsb{t}$ is $(-1,1,0)$.
\end{itemize}
\end{Exa}

A consequence of Remark \ref{rev_01_rem_reverse}.2 is:
\begin{Rem}$ $
\label{heredit}
If $\bsb{s}$ and $\bsb{t}$ are close subexcedant sequences
and $m$ is an integer such that both $\bsb{u}=\mathrm{first}(m;0;\bsb{s})$ 
and $\bsb{v}=\mathrm{first}(m;0;\bsb{t})$ exist,
then $\bsb{u}$ and $\bsb{v}$ are also close.
\end{Rem}

\begin{The}
\label{main_th}
Two consecutive sequences in $S(k,n)$ generated by the algorithm
{\tt Gen1\_Gray} are close. 
\end{The}
\begin{proof}

Let $\bsb{s}$ and $\bsb{t}$ be two consecutive sequences 
generated by the call of {\tt Gen1\_Gray($k$,$n$,$0$)}.
Then there is a sequence $\bsb{c}=c_1c_2\ldots c_n$ and a
recursive call of {\tt Gen1\_Gray} acting on $\bsb{c}$
(referred later as the {\it root call} for $\bsb{s}$ and $\bsb{t}$) which produces,
in the {\tt for} statement, two calls so that $\bsb{s}$ is the 
last sequence produced by the first of them and $\bsb{t}$
the first produced by the second of them.  

By Remark \ref{rev_01_rem_reverse}.1 it is enough to prove that $\bsb{s}$ 
and $\bsb{t}$ are close when
their root call has direction $0$. 

Let $\ell$ and $r$, $\ell\neq r$, be the leftmost and the rightmost increasable positions
in $\bsb{c}$ (and so $c_1=c_2=\ldots =c_{r-1}=0$, and possibly $c_r=0$);
and $i$ and $i+1$ be the positions where
$\bsb{c}$ is modified by the root call in order to produce
eventually $\bsb{s}$ and $\bsb{t}$. Also we denote $m=k-\sum_{j=1}^n c_j$ and
$e=m-\frac{\ell\cdot (\ell-1)}{2}$.

We will give the shape of $\bsb{s}$ and $\bsb{t}$ according to
the following four cases.

\begin{enumerate}
\item $i=\ell$ and $r-\ell$ is even,
\item $i=\ell$ and $r-\ell$ is odd,
\item $i\neq\ell$ and the call corresponding to $i$ in the {\tt for}
      statement of the root call has direction $0$
      (and so that corresponding to $i+1$ has direction $1$),
\item $i\neq\ell$ and the call corresponding to $i$ in the {\tt for}
      statement of the root call has direction $1$
      (and so that corresponding to $i+1$ has direction $0$).
\end{enumerate}   

\noindent
Case 1.
\begin{eqnarray*}
\bsb{s} & = & \mathrm{last}(m-e;0;00\ldots e c_{\ell+1}\ldots c_n)\\
        & = & \mathrm{first}(m-e;1;00\ldots e c_{\ell+1}\ldots c_n)\\
	 & = & \left\{ \begin {array}{lcc}
                       \mathrm{first}  (m-e-(\ell-2);0;00\ldots (\ell-2)ec_{\ell+1} \ldots c_n) & {\rm if} & e=\ell-1\\
                       \mathrm{first}  (m-e-(\ell-2);0;00\ldots
		       (\ell-3)(e+1)c_{\ell+1}\ldots c_n) &    {\rm if} &
		       e<\ell-1,
                      \end {array}
                    \right.
\end{eqnarray*}
and
\begin{eqnarray*}
\bsb{t}  & = & \mathrm{first} (m-1;0;00\ldots (c_{\ell+1}+1)\ldots c_n)\\
         & = & \mathrm{first} (m-e;0;00\ldots (e-1)(c_{\ell+1}+1)\ldots c_n)\\
         & = & \mathrm{first} (m-e-(\ell-2);0;00\ldots (\ell-2)(e-1)(c_{\ell+1}+1)\ldots c_n). 
\end{eqnarray*}

\noindent
Case 2. In this case $\bsb{s}$ is the same as in the previous case and

\begin{eqnarray*}
\bsb{t} &
= & \mathrm{first}(m-1;1;00\ldots 0(c_{\ell+1}+1)\ldots c_n)  \\
&
= & \left\{ \begin {array}{lcc}
\mathrm{first}  (m-2;0;00\ldots 0     (c_{\ell+1}+2)\ldots c_n) & {\rm if} & c_{\ell+1}+2\leq \ell\\
\mathrm{first}  (m-e;0;00\ldots 0(e-1)(c_{\ell+1}+1)\ldots c_n) & {\rm if} & c_{\ell+1}+2>\ell
\end {array}
\right.\\
  & 
= & \left\{ \begin {array}{lcc}
\mathrm{first}  (m-e-(\ell-2);0;00\ldots 0(\ell-2)(e-2)(c_{\ell+1}+2)\ldots c_n) & {\rm if} & c_{\ell+1}+2\leq \ell\\
\mathrm{first}  (m-e-(\ell-2);0;00\ldots (\ell-2)(e-1)(c_{\ell+1}+1)\ldots c_n) &
{\rm if} & c_{\ell+1}+2>\ell.
\end {array}
\right.
\end{eqnarray*}

\noindent
Case 3. In this case $c_i=0$ and

\begin{eqnarray*}
\bsb{s}  & = & \mathrm{last} (m-1;0;00\ldots  01c_{i+1}\ldots c_n)\\
         & = & \mathrm{last} (m-2;1;00\ldots  02c_{i+1}\ldots c_n)\\ 
         & = & \mathrm{first} (m-2;0;00\ldots  02c_{i+1}\ldots c_n),	 
\end{eqnarray*}
and
\begin{eqnarray*}
\bsb{t}  & = & \mathrm{first} (m-1;1;00\ldots  0(c_{i+1}+1)\ldots c_n)\\
         & = & 
\left\{ \begin {array}{lcc}
\mathrm{first}  (m-2;0;00\ldots 0(c_{i+1}+2)\ldots c_n) & {\rm if} & c_{i+1}+2\leq i\\
\mathrm{first}  (m-2;0;00\ldots 1(c_{i+1}+1)\ldots c_n) & {\rm if} & c_{i+1}+2>    i.
\end {array}
\right.	 
\end{eqnarray*}

\noindent
Case 4. As previously, $c_i=0$ and

\begin{eqnarray*}
\bsb{s}  & = & \mathrm{last} (m-1;1;00\ldots  01c_{i+1}\ldots c_n)\\
         & = & \mathrm{first} (m-1;0;00\ldots 01c_{i+1}\ldots c_n), 
\end{eqnarray*}
and

$$\bsb{t}=\mathrm{first} (m-1;0;00\ldots 00(c_{i+1}+1)\ldots c_n).
$$
Finally, by Remark \ref{heredit} it follows that in each of the 
four cases $\bsb{s}$ and $\bsb{t}$ 
are close, and the statement holds.
\end{proof}

As a byproduct of the previous theorem and Remark \ref{rem_2_points}.2 we have

\begin{Rem}
\label{boure}
If $\bsb{s}=s_1s_2\ldots s_n$ and 
$\bsb{t}=t_1t_2\ldots t_n$ are two consecutive sequences generated by 
{\tt Gen1\_Gray} and $p$ is the rightmost position where  
they differ, then 
$s_1s_2\ldots s_{p-2}$ and $t_1t_2\ldots t_{p-2}$ are the smallest, 
in co-lex order, sequences in $S(x,p-2)$ and $S(y,p-2)$, 
respectively, with $x=s_1+s_2+\ldots +s_{p-2}$ and 
$y=t_1+t_2+\ldots +t_{p-2}$.
Remark that $s_1s_2\ldots s_{p-2}=t_1t_2\ldots t_{p-2}$,
and so $x=y$, if $\bsb{s}$ and $\bsb{t}$ differ in two (adjacent) positions.
\end{Rem}

\subsection{Algorithm {\tt Gen2\_Gray}}

\begin{figure}
\begin{center}
{
\psset{levelsep=18mm,treesep=3mm,radius=-0.5mm}
\pstree[nodesep=1mm,treemode=R]{\TR{$\phantom{000}\cdots$}}
{ \pstree{\TR{  $0000401$}}
  { \pstree{\TR{ $0003401$}}
    { \pstree{\TR{ $0023401$}}
      {\pstree{\TR{ $0123401$}}{}
      }
    }
  }
}
}
\end{center}
\caption{\label{q_terminal_n}
Four successive q-terminal nodes in the generating tree
induced by the call {\tt Gen1\_Gray}(11,7,0) which generates
the list $\mathcal{S}(11,7)$.}
\end{figure}

Since the generating tree induced by the call of {\tt Gen1\_Gray} 
contains still arbitrary length branches of 
nodes of degree one it has a poor time complexity.
Here we show how some of these nodes can be avoided in order to obtain the 
efficient generating algorithm {\tt Gen2\_Colex}
presented in Figure \ref{Algo_Gen2Gray}.

A {\it quasi-terminal node} ({\it q-terminal node} for short)
in the tree induced by a generating algorithm
is defined recursively as:
a q-terminal node is either a terminal node 
(node with no successor) or a node with only
one successor which in turn is a q-terminal node.
The q-terminal nodes occur for the calls of {\tt Gen1\_Gray($k,r,dir$)}
when $k=\frac{r(r-1)}{2}$.
See Figure~\ref{q_terminal_n} for an example.

The key improvement made by {\tt Gen2\_Gray} consists in  its last 
parameter $p$, which gives 
the rightmost position where the current sequence differ from its previous 
one in the list $\mathcal{S}(k,n)$, and {\tt Gen2\_Gray}  
stops the recursive calls of more than three
successive q-terminal calls.
Thus, {\tt Gen2\_Gray} generates
only suffixes of the form $c_{p-2}c_{p-1}c_{p}\ldots c_n$;
see Table \ref{list_pref} for an example.
Since two consecutive sequences in the Gray code $\mathcal{S}(k,n)$
differ in at most three adjacent positions, these suffixes 
are enough to generate efficiently $\mathcal{S}(k,n)$,
and to generate (in Section \ref{gen_perms}) a Gray code for the set of length 
$n$ permutations having the major index equal to $k$.

Now we explain how the parameter $p$ propagates through recursive calls.
A non terminal call of {\tt Gen2\_Gray} produces one or several calls. 
The first of them (corresponding to a left child
in the generating tree) inherits the value of the parameter 
$p$ from its parent call; in the other calls the value of this parameter 
is the rightmost position where the current sequence
differs from its previous generated one; this value is $i$ if $dir=0$ and 
$i+1$ if $dir=1$. So, each call keeps in the last 
parameter $p$ the rightmost position where the current generated sequence 
differs from its previous one in the list ${\mathcal S}(k,n)$.
Procedure {\tt Gen2\_Gray} prevents to produce more than three successive
q-terminal calls. For convenience, initially $p=0$.

The last two parameters $p$ and $u$ of procedure {\tt Gen2\_Gray} and
output by it are used by procedure 
{\tt Update\_Perm} in Section \ref{gen_perms} in order to generates 
permutations with a given major index; $u$ keeps the value
of $c_1+c_2+\ldots +c_p$, and for convenience, initially $u=0$.

\medskip
Even we will not make use later we 
sketch below an algorithm for efficiently
generating the list ${\mathcal S}(k,n)$:
   
\begin{itemize}
\item initialize $\bsb{d}$ by the first sequence in $\mathcal{S}(k,n)$,
      i.e, the the smallest sequence in $S(k,n)$
      in co-lex order, or equivalently, the largest one in lexicographical orders,
      and  $\bsb{c}$ by $0\,0\,\ldots\, 0$,
\item run {\tt Gen2\_Gray($k,n,0,0,0)$} and  for each $p$ output by it
      update $\bsb{d}$ as: $d[p-2]:=c[p-2]$, $d[p-1]:=c[p-1]$, $d[p]:=c[p]$. 
\end{itemize}

\begin{figure}
\begin{tabular}{|c|}
\hline
\small\tt
\begin{minipage}[c]{0.460\linewidth}
\begin{tabbing}
\hspace{1.cm}\=\hspace{0.9cm}\=\hspace{0.7cm}\=\hspace{1.9cm}
            \=\hspace{1.5cm}\kill
procedure Gen2\_Gray($k$,$r$,$dir$,$p$,$u$)\\
global $n,c,b$;\\
if $k=0$ or $(p-r)\geq 3$ and $k=\frac{r(r-1)}{2}$\\
then output($p,u$);\\
else \> if $c[r]=r-1$ \\
     \>then $r:=r-1$; \\
     \>end if\\     
     \> $\ell:=\min\{s\,|\,\frac{s(s-1)}{2} \geq k\}$;\\
     \> $e:=k-\frac{(\ell-1)(\ell-2)}{2}$;\\ 
     \> if $dir=0$ \\
     \> then \> $c[\ell]:=c[\ell]+e$; Gen2\_Gray($k-e$,$\ell$,$0$,$p$,$u$);
     $c[\ell]:=c[\ell]-e$;\\
     \>      \> $dir:=(r-\ell) \mod 2$;\\
     \>      \> for $i:=\ell+1$ to $r$ do\\
     \>      \>  \>$c[i]:=c[i]+1$; Gen2\_Gray($k-1$,$i$,$dir$,$i$,$k-1+c[i]$); $dir:=(dir+1)\mod 2$; $c[i]:=c[i]-1$;\\
     \>      \>  end do\\
     \> else \> $dir:=0$;\\
     \>      \> for $i:=r$ downto $\ell+1$ do\\  
     \>      \> \> if $i=r$ then $q:=p$; $v:=u$; else $q=:i+1$; $v:=c[i+1]+k$; end if\\
     \>      \> \> $c[i]:=c[i]+1$; Gen2\_Gray($k-1$,$i$,$dir$,$q$,$v$); $dir:=(dir+1)\mod 2$; $c[i]:=c[i]-1$;\\
     \>      \>  end do\\
     \>      \> if $\ell=r$ then $q:=p$; $v:=u$; else $q:=\ell+1$; $v:=c[\ell+1]+k$; end if\\    
     \>      \> $c[\ell]:=c[\ell]+e$; Gen2\_Gray($k-e$,$\ell$,$1$,$q$,$v$);
     $c[\ell]:=c[\ell]-e$;\\
     \> end if\\
end if\\
end procedure.
\end{tabbing}
\end{minipage}
\\ \hline
\end{tabular}
\caption{\label{Algo_Gen2Gray}
Algorithm  {\tt Gen2\_Gray}.}
\end{figure}


\begin{table}
\begin{center}
\begin{tabular}{|r|c|c||r|c|c|}
\hline
sequence & $p$ & permutation & sequence & $p$ & permutation\\
\hline
$0\,1\,2\,1\,0\,0$ &     & $2\,1\,4\,3\,5\,6$ &  $0\,1\,\underline{0\,0\,1}\,2$ & $5$ & $5\,3\,6\,1\,2\,4$ \\
$0\,\underline{1\,0\,3}\,0\,0$ & $4$ & $3\,2\,4\,1\,5\,6$ &  $\underline{0\,0\,1}\,0\,1\,2$ & $3$ & $6\,3\,5\,1\,2\,4$ \\
$\underline{0\,0\,1}\,3\,0\,0$ & $3$ & $4\,2\,3\,1\,5\,6$ &  $0\,\underline{0\,0\,1}\,1\,2$ & $4$ & $1\,3\,5\,6\,2\,4$ \\
$0\,\underline{0\,2\,2}\,0\,0$ & $4$ & $4\,1\,3\,2\,5\,6$ &  $0\,0\,\underline{0\,0\,2}\,2$ & $5$ & $2\,3\,5\,6\,1\,4$ \\
$\underline{0\,1\,1}\,2\,0\,0$ & $3$ & $3\,1\,4\,2\,5\,6$ &  $0\,0\,0\,\underline{0\,0\,4}$ & $6$ & $3\,4\,5\,6\,1\,2$ \\     
$0\,1\,\underline{2\,0\,1}\,0$ & $5$ & $2\,1\,5\,3\,4\,6$ &  $0\,0\,0\,\underline{0\,1\,3}$ & $6$ & $2\,4\,5\,6\,1\,3$\\
$0\,\underline{1\,1\,1}\,1\,0$ & $4$ & $3\,1\,5\,2\,4\,6$ &  $0\,0\,\underline{0\,1\,0}\,3$ & $5$ & $1\,4\,5\,6\,2\,3$ \\
$\underline{0\,0\,2}\,1\,1\,0$ & $3$ & $5\,1\,3\,2\,4\,6$ &  $0\,\underline{0\,1\,0}\,0\,3$ & $4$ & $6\,4\,5\,1\,2\,3$ \\
$0\,\underline{0\,0\,3}\,1\,0$ & $4$ & $1\,2\,3\,5\,4\,6$ &  $\underline{0\,1\,0}\,0\,0\,3$ & $3$ & $5\,4\,6\,1\,2\,3$ \\
$0\,\underline{0\,1\,2}\,1\,0$ & $4$ & $5\,2\,3\,1\,4\,6$ &  $0\,1\,0\,\underline{0\,2\,1}$ & $6$ & $4\,3\,6\,1\,2\,5$ \\   
$\underline{0\,1\,0}\,2\,1\,0$ & $3$ & $3\,2\,5\,1\,4\,6$ &  $\underline{0\,0\,1}\,0\,2\,1$ & $3$ & $6\,3\,4\,1\,2\,5$ \\    
$0\,1\,\underline{0\,0\,3}\,0$ & $5$ & $4\,3\,5\,1\,2\,6$ &  $0\,\underline{0\,0\,1}\,2\,1$ & $4$ & $1\,3\,4\,6\,2\,5$ \\ 
$\underline{0\,0\,1}\,0\,3\,0$ & $3$ & $5\,3\,4\,1\,2\,6$ &  $0\,0\,\underline{0\,0\,3}\,1$ & $5$ & $2\,3\,4\,6\,1\,5$ \\  
$0\,\underline{0\,0\,1}\,3\,0$ & $4$ & $1\,3\,4\,5\,2\,6$ &  $0\,0\,\underline{0\,2\,1}\,1$ & $5$ & $1\,2\,4\,6\,3\,5$ \\ 
$0\,0\,\underline{0\,0\,4}\,0$ & $5$ & $2\,3\,4\,5\,1\,6$ &  $0\,\underline{0\,1\,1}\,1\,1$ & $4$ & $6\,2\,4\,1\,3\,5$  \\
$0\,0\,\underline{0\,2\,2}\,0$ & $5$ & $1\,2\,4\,5\,3\,6$ &  $\underline{0\,1\,0}\,1\,1\,1$ & $3$ & $4\,2\,6\,1\,3\,5$ \\
$0\,\underline{0\,1\,1}\,2\,0$ & $4$ & $5\,2\,4\,1\,3\,6$ &  $0\,\underline{0\,2\,0}\,1\,1$ & $4$ & $6\,1\,4\,2\,3\,5$ \\
$\underline{0\,1\,0}\,1\,2\,0$ & $3$ & $4\,2\,5\,1\,3\,6$ &  $\underline{0\,1\,1}\,0\,1\,1$ & $3$ & $4\,1\,6\,2\,3\,5$\\
$0\,\underline{0\,2\,0}\,2\,0$ & $4$ & $5\,1\,4\,2\,3\,6$ &  $0\,1\,\underline{1\,1\,0}\,1$ & $5$ & $3\,1\,6\,2\,4\,5$ \\
$\underline{0\,1\,1}\,0\,2\,0$ & $3$ & $4\,1\,5\,2\,3\,6$ &  $\underline{0\,0\,2}\,1\,0\,1$ & $3$ & $6\,1\,3\,2\,4\,5$ \\
$0\,1\,1\,\underline{0\,0\,2}$ & $6$ & $5\,1\,6\,2\,3\,4$ &  $0\,\underline{0\,0\,3}\,0\,1$ & $4$ & $1\,2\,3\,6\,4\,5$\\  
$\underline{0\,0\,2}\,0\,0\,2$ & $3$ & $6\,1\,5\,2\,3\,4$ &  $0\,\underline{0\,1\,2}\,0\,1$ & $4$ & $6\,2\,3\,1\,4\,5$\\
$0\,\underline{0\,0\,2}\,0\,2$ & $4$ & $1\,2\,5\,6\,3\,4$ &  $\underline{0\,1\,0}\,2\,0\,1$ & $3$ & $3\,2\,6\,1\,4\,5$ \\
$0\,\underline{0\,1\,1}\,0\,2$ & $4$ & $6\,2\,5\,1\,3\,4$ &  $0\,\underline{1\,2\,0}\,0\,1$ & $4$ & $2\,1\,6\,3\,4\,5$ \\ 
$\underline{0\,1\,0}\,1\,0\,2$ & $3$ & $5\,2\,6\,1\,3\,4$ &                  &  &  \\ 
  \hline
\end{tabular}
\end{center}
\caption{
\label{list_pref}The subexcedant sequences generated by the call of {\tt Gen1\_Gray($4,6,0$)}
and their corresponding length $6$ permutations with major index equals $4$,
permutations descent set is either $\{1,3\}$ or $\{4\}$.
The three leftmost entries ($c_{p-2}$,$c_{p-1}$,$c_p$) updated by the call of 
{\tt Gen2\_Gray($4,6,0,0,0$)} are underlined, where
$p$ is the rightmost position  where a subexcedant
sequence differ from its predecessor.
}
\end{table}

\subsubsection*{Analyze of {\tt Gen2\_Gray}}

For a call of {\tt Gen2\_Gray($k$,$r$,$dir$,$p$,$u$)}
necessarily $k\leq\frac{r(r-1)}{2}$, and 
if $k>0$ and
\begin{itemize}
\item $k\leq \frac{(r-1)(r-2)}{2}$, then this call
      produces at least two recursive calls,
\item $\frac{(r-1)(r-2)}{2}<k<\frac{r(r-1)}{2}$,
      then this call produces a unique recursive call
      (of the form {\tt Gen2\_Gray($k'$,$r$,$\cdot$,$\cdot$,$\cdot$)},
      with $k'=k-\frac{(r-1)(r-2)}{2}$),
      which in turn produce two calls,
\item $k=\frac{r(r-1)}{2}$, then this call is q-terminal call.
\end{itemize}
Sine the procedure {\tt Gen2\_Gray} stops after three successive 
q-terminal calls, with a slight modification of Ruskey and van Baronaigien's \cite{Roe_93}
{\it `CAT'} principle (see also \cite{Rus_00}) it follows that {\tt Gen2\_Gray} runs in constant 
amortized time.

\section{The McMahon code of a permutation}
\label{sec_Mc_code}

Here we present the bijection 
$\psi:S(n)\rightarrow \SS_n$, introduced in \cite{Vaj_11}, which have the following properties:
\begin{itemize}
\item the image through $\psi$ of $S(k,n)$ is the set of permutations in $\SS_n$ with 
major index $k$,
\item $\psi$ is a `Gray code preserving bijection' (see Theorem \ref{sigma_t_a}),
\item $\tau$ is easily computed from $\sigma$ and from the difference between $\bsb{s}$ and $\bsb{t}$,
the McMahon code of  $\sigma$ and $\tau$, if $\bsb{s}$ and $\bsb{t}$ are close.
\end{itemize}
In the next section we apply $\psi$ in order
to construct a list for the permutations in $\SS_n$ with a major index 
equals $k$ from the Gray code list $\mathcal{S}(k,n)$.

Let permutations act on indices, i.e., for 
$\sigma=\sigma_1\,\sigma_2\, \ldots \,\sigma_n$ and 
$\tau=\tau_1\,\tau_2\, \ldots \,\tau_n$
two permutations in $\SS_n$, 
$\sigma\cdot\tau=\sigma_{\tau_1}\,\sigma_{\tau_2}\, \ldots \,\sigma_{\tau_n}$.
For a fixed integer $n$, let $k$ and $u$ be two integers, $0\leq k<u\leq n$,
and define
$\llsq u,k \rrsq\in\SS_n$ as the permutation obtained after 
$k$ right circular shifts of the length-$u$ prefix of the identity
in $\SS_n$. 
In two line notation

$$
\llsq u,k \rrsq=
\left(
\begin{array}{cccccccccc}
1     & 2   & \cdots & k & k+1   & \cdots   & u   & u+1     &\cdots & n   \\
u-k+1 & u-k+2 & \cdots & u & 1   &\cdots    & u-k & u+1     &\cdots & n 
\end{array}
\right).
$$

For example, in $\SS_5$ we have: 
$\llsq 3,1\rrsq=\underline{3\,1\,2}\,4\,5$, 
$\llsq 3,2\rrsq=\underline{2\,3\,1}\,4\,5$  and
$\llsq 5,3\rrsq=\underline{3\,4\,5\,1\,2}$ (the rotated elements are underlined).

\medskip
Let 
$\psi:S(n)\rightarrow \SS_n
$
be the function defined by 

\begin{equation}
\label{def_psi}
\begin{array}{ccl}
\psi(t_1t_2\ldots t_n) 
  & = & \llsq n,t_n\rrsq\cdot \llsq n-1,t_{n-1}\rrsq\cdot\ldots\cdot \llsq i,t_i\rrsq\cdot  
      \ldots \cdot\llsq 2,t_2\rrsq\cdot \llsq 1,t_1 \rrsq \\
  & = & \displaystyle \prod_{i=n}^1\llsq i,t_i\rrsq.
\end{array}
\end{equation}

\begin{Lem}[\cite{Vaj_11}]$ $
\begin{enumerate}
\item 
The function $\psi$ defined above is a bijection.
\item
For every $\bsb{t}=t_1t_2\ldots t_n\in S(n)$, we have 
$\maj \prod_{i=n}^1\llsq i,t_i\rrsq=\sum_{i=1}^nt_i$.
\end{enumerate}
\end{Lem}

The first point of the previous lemma says that every permutation 
$\pi\in\SS_n$ can be uniquely written as
$\prod_{i=n}^1\llsq i,t_i\rrsq$ for some $t_i$'s, and 
the subexcedant sequence $t_1t_2\ldots t_n$ is called the 
{\it McMahon code} of $\pi$. 
As a consequence of the second point of this lemma we have:

\begin{Rem}
The restriction of $\psi$ maps bijectively permutations in 
$S(k,n)$ into permutations in $\SS_n$ with major index equals $k$.
\end{Rem}

\begin{Exa}
The permutation $\pi =5\,2\,1\,6\,4\,3\in \SS_n$ can be obtained from
the identity by the following prefix rotations:
$$1\,2\,3\,4\,5\,6
  \overset{\llsq 6,3 \rrsq}{\longrightarrow}
  4\,5\,6\,1\,2\,3
  \overset{\llsq 5,4 \rrsq}{\longrightarrow}
  5\,6\,1\,2\,4\,3
  \overset{\llsq 4,2 \rrsq}{\longrightarrow}
  1\,2\,5\,6\,4\,3
  \overset{\llsq 3,2 \rrsq}{\longrightarrow}
  2\,5\,1\,6\,4\,3
  \overset{\llsq 2,1 \rrsq}{\longrightarrow}
  5\,2\,1\,6\,4\,3
  \overset{\llsq 1,0 \rrsq}{\longrightarrow}
  5\,2\,1\,6\,4\,3,
$$
so 

$$
\pi=
\llsq 6,3\rrsq\cdot\llsq 5,4 \rrsq\cdot\llsq 4,2 \rrsq\cdot\llsq 3,2 
\rrsq\cdot\llsq 2,1 \rrsq\cdot\llsq 1,0 \rrsq,$$
and thus
$$
\maj\ \pi =3+4+2+2+1+0=12.
$$
\end{Exa}

Theorem \ref{sigma_t_a} below states that if two permutations have 
their McMahon code differing in two adjacent positions, and by 
$1$ and $-1$ in these positions, then these permutations differ by the transposition 
of two entries.
Before proving this theorem we need the following 
two propositions,
where the transposition $\langle u, v\rangle$  denote the permutation 
$\pi$ (of convenient length) with $\pi(i)=i$ for all $i$, except $\pi(u)=v$
and $\pi(v)=u$.

\begin{Pro}
\label{first_trans}
Let $n,u$ and $v$ be three integers,
$n\geq 3$, $0\leq u\leq n-2$, $1\leq v\leq n-2$, and
$\sigma,\tau\in\SS_n$ defined by:

\begin{itemize}
\item $\sigma= \llsq n,u\rrsq\  \cdot \llsq n-1,v\rrsq$, and
\item $\tau  = \llsq n,u+1\rrsq \cdot \llsq n-1,v-1\rrsq$.
\end{itemize}
Then 
$$
\tau=\sigma\cdot\langle n, v\rangle.
$$
\end{Pro}
\begin{proof}
First, remark that:

\begin{itemize}
\item $\llsq n,u+1\rrsq$, is a right circular shift of $\llsq n,u\rrsq$, and
\item $\llsq n-1,v-1\rrsq$ is a left circular shift of 
the first $(n-1)$ entries of $\llsq n-1,v\rrsq$,
\end{itemize}
and so $\sigma(i)=\tau(i)$ for all $i$, $1\leq i\leq n$, except for 
$i=n$ and $i=v$.
\end{proof}

\begin{Exa}
For $n=7$, $u=4$ and $v=3$ we have
\begin{itemize}
\item $\sigma=\llsq n,u  \rrsq\cdot \llsq n-1,v  \rrsq=\llsq 7,4\rrsq\cdot \llsq
6,3\rrsq=7\,1\,2\,4\,5\,6\,3$,
\item $\tau=  \llsq n,u+1\rrsq\cdot \llsq n-1,v-1\rrsq=\llsq 7,5\rrsq\cdot \llsq 6,2\rrsq=
7\,1\,3\,4\,5\,6\,2$,
\item $\langle n,v\rangle=\langle 7,3\rangle$, 
\end{itemize}
and $\tau=\sigma\cdot \langle n,v\rangle$.
\end{Exa}

\begin{Pro}
\label{before_th}
If $\pi\in\SS_n$ and $\langle u,v\rangle$ is a transposition in 
$\SS_n$, then 

$$\pi^{-1}\cdot \langle u,v\rangle\cdot\pi=
\langle\pi^{-1}(u),\pi^{-1}(v)\rangle.$$
\end{Pro}
\begin{proof}
Indeed, $(\pi^{-1}\cdot \langle u,v\rangle\cdot\pi)(i)=i$,
for all $i$, except for $i=\pi^{-1}(u)$ and $i=\pi^{-1}(v)$.
\end{proof}

\begin{The}
\label{sigma_t_a}
Let $\sigma$ and $\tau$ be two permutations in $\SS_n $, $n\geq 3$, and 
$\bsb{s}=s_1s_2\ldots s_n$ and $\bsb{t}=t_1t_2\ldots t_n$
their McMahon codes. If there is a $f$, $2\leq f\leq n-1$
such that $t_i=s_i$ for all $i$, except $t_f=s_f-1$
and $t_{f+1}=s_{f+1}+1$, then $\tau$ and $\sigma$ differ by a 
transposition. More precisely, 
$$
\tau=\sigma \cdot \langle \alpha^{-1}(u), \alpha^{-1}(v)\rangle
$$
where 
$$
\alpha=\prod_{i=f-1}^{1}\llsq i,s_i  \rrsq=\prod_{i=f-1}^{1}\llsq i,t_i  \rrsq,
$$
and $u=f+1$, $v=s_f$.
\end{The}
\begin{proof} $ $
\begin{itemize}
\item 
$\tau=\prod_{i=n}^{1}\llsq i,t_i \rrsq$, and so 
$\tau\cdot\alpha^{-1}=\prod_{i=n}^{f}\llsq i,t_i \rrsq$, and
\item
$\sigma=\prod_{i=n}^{1}\llsq i,s_i \rrsq$, and 
$\sigma\cdot\alpha^{-1}=\prod_{i=n}^{f}\llsq i,s_i \rrsq$.

\end{itemize}

But, by Proposition \ref{first_trans}, 

$$
\prod_{i=n}^{f}\llsq i,t_i \rrsq=
\prod_{i=n}^{f}\llsq i,s_i \rrsq\cdot
\langle f+1,s_f\rangle 
$$
or, equivalently

$$
\tau\cdot\alpha^{-1}=\sigma\cdot\alpha^{-1}\cdot\langle f+1,s_f\rangle,
$$
and by Proposition \ref{before_th}, the results holds.
\end{proof}

The previous theorem says that $\sigma$ and $\tau$
`have a small difference' provided that their McMahon code, 
$\bsb{s}$ and $\bsb{t}$, do so. Actually, we need that 
$\bsb{s}$ and $\bsb{t}$ are consecutive sequences 
in the list $\mathcal{S}(k,n)$
and they have a more particular shape (see Remark \ref{boure}).
In this context, permutations having minimal
McMahon code play a particular role.

It is routine to check the following proposition (see Figure \ref{Fig_3}
for an example).
\begin{Pro} 
\label{alpha_n_k}
Let $n$ and $k$ be two integers, $0<k\leq\frac{n(n-1)}{2}$;
$\bsb{a}=a_1a_2\ldots a_n$ be the smallest subexcedant sequence in co-lex order
with $\sum_{i=1}^n a_i=k$, and 
$\alpha=\alpha_{n,k}=\psi(\bsb{a})$ be the permutation in $\SS_n$ having
its McMahon code $\bsb{a}$.
Let  $j=\max\, \{ i : a_i\neq 0\}$, that is, $\bsb{a}$ has the form
$$ 012\ldots (j-3)(j-2)a_j00\ldots 0.$$
Then
\begin{equation}
\label{def_alpha}
\alpha(i)=\left\{ \begin {array}{ccc}
j-a_j-i & {\rm if} & 1\leq i\leq j-(a_j+1), \\
2j-a_j-i & {\rm if} & j-(a_j+1)<i\leq j, \\
i & {\rm if} & i>j.
\end {array}
\right.
\end{equation}
\end{Pro} 

\begin{figure}
\centering
\unitlength=3mm
\hspace{4mm}

\unitlength=3mm
\begin{picture}(15,16)
\put(0.5,1.5){\vector(0,1){9.5}}
\put(0.5,1.5){\vector(1,0){9.5}}
\put(8.5,1.5){\line(0,1){8}}

\put(0.5,9.5){\line(1,0){8}}



\put(1,3){\circle*{0.4}} 
\put(2,2){\circle*{0.4}} 
\put(3,7){\circle*{0.4}} 
\put(4,6){\circle*{0.4}} 
\put(5,5){\circle*{0.4}} 
\put(6,4){\circle*{0.4}} 
\put(7,9){\circle*{0.4}} 
\put(8,8){\circle*{0.4}} 

\put(7,0.){ $i$}
\put(-2.5,8){ $\alpha(i)$}
\end{picture}
\caption{
The permutation $\alpha=2\,1\,6\,5\,4\,3\,8\,7$  with the McMahon code 
$\bsb{a}=0\,1\,2\,3\,4\,3\,0\,0$, 
the the smallest, in co-lex order, subexcedant sequence in $S(13,8)$,
see Proposition \ref{alpha_n_k}.}
\label{Fig_3}
\end{figure}
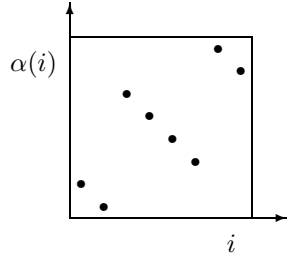

\begin{Rem}
\label{rem_inv} 
The permutation $\alpha$ defined in Proposition in \ref{alpha_n_k}
is an involution, that is $\alpha^{-1}=\alpha$.

\end{Rem} 

Combining Proposition \ref{alpha_n_k} and Remark \ref{rem_inv},
Theorem \ref{sigma_t_a} becomes in particular

\begin{Pro}
\label{combi}
Let $\sigma$, $\tau$, $\bsb{s}$ and $\bsb{t}$ be as in Theorem \ref{sigma_t_a}.
In addition, let suppose that there is a $j$, $0\leq j\leq f-1$, such that 
\begin{itemize}
\item[1.] $s_i=t_i=0$ for $j<i\leq f-1$, and
\item[2.] if $j>0$, then 
    \begin{itemize}
    \item $s_j=t_j\neq 0$, and 
    \item $s_i=t_i=i-1$ for $1\leq i<j$.
    \end{itemize}
\end{itemize}
Then
$$
\tau=\sigma \cdot \langle \phi_j(f+1), \phi_j(s_f)\rangle
$$
with 
\begin{equation}
\label{phi}
\phi_j(i)=\left\{ \begin {array}{ccc}
j-s_j-i & {\rm if} & 1\leq i\leq j-(s_j+1), \\
2j-s_j-i & {\rm if} & j-(s_j+1)<i\leq j, \\
i & {\rm if} & i>j.
\end {array}
\right.
\end{equation}
\end{Pro}

\noindent
Notice that, the conditions 1 and 2 in the previous proposition 
require that $s_1s_2\ldots s_{f-1}=t_1t_2\ldots t_{f-1}$
be the smallest subexcedant sequence, in co-lex order, in $S(f-1)$ with
fixed value for $\sum_{i=1}^{f-1}s_i=\sum_{i=1}^{f-1}t_i$.
Also, for point 2, necessarily $j\geq 2$.

%
%
%

\section{Generating permutations with a given major index}
\label{gen_perms}

Let $\sigma$ and $\tau$ be two permutations 
with their McMahon code $\bsb{s}=s_1s_2\ldots s_n$
and  $\bsb{t}=t_1t_2\ldots t_n$ belonging to $S(k,n)$, and 
differing in positions $f$ and $f+1$ by $1$ and $-1$ in 
these positions.

Let 
\begin{itemize}
\item $v=s_f-t_f\in \{-1,1\}$, and 
\item $x=\sum_{i=1}^{f-1}s_i=\sum_{i=1}^{f-1}t_i$.
\end{itemize}

If $s_1s_2\ldots s_{f-1}$ is the smallest sequence in 
$S(x,f-1)$, in co-lex order, then applying Proposition \ref{combi}
it follows that the run of the procedure {\tt transp($v,f,x$)} defined in 
Figure~\ref{algos_transp} transforms $\sigma$ into $\tau$ and $\bsb{s}$ into $\bsb{t}$.

\begin{figure}[h]
\begin{center}
{\tt
\begin{tabular}{|c|}
\hline
\begin{minipage}[c]{.46\linewidth}
\begin{tabbing}
\hspace{0.75cm}\=\hspace{1.cm}\=\hspace{0.5cm}\=\hspace{1.6cm}\=\hspace{1.9cm}
            \=\hspace{1.5cm}\kill
procedure transp($v,f,x$)\\
\> $j:=\min \{i\,:\, \frac{i(i-1)}{2}\geq x\}$;\\
\> if $v=1$\\
\> then $\sigma:=\sigma\cdot \langle\phi_j(f+1),\phi_j(s[f]) \rangle$;\\
\> else $\sigma:=\sigma\cdot \langle\phi_j(f+1),\phi_j(s[f]+1) \rangle$;\\
\> endif\\
\> $s[f]:=s[f]-v$;\\
\> $s[f+1]:=s[f+1]+v$;\\
end procedure.
\end{tabbing}
\end{minipage}\\
\hline
\end{tabular}
}
\caption{\label{algos_transp} Algorithm {\tt transp}, where $\phi_j$
         is defined in relation (\ref{phi}).}
\end{center}
\end{figure}

Let now $f$ be the leftmost position where two consecutive 
sequences $\bsb{s}$ and $\bsb{t}$ in the list $\mathcal{S}(k,n)$
differ, and $\sigma$ and $\tau$ be the permutations 
having they McMahon code $\bsb{s}$ and $\bsb{t}$.
By Remarks \ref{rem_inter} and \ref{boure} we have that,
repeated calls of {\tt transp} 
transform $\bsb{s}$ into $\bsb{t}$, and $\sigma$ into $\tau$.
This is true for each possible 
$3$-tuples given in Definition~\ref{3_tuple} and 
corresponding to two consecutive 
subexcedant sequences in $\mathcal{S}(k,n)$, and
algorithm {\tt Update\_Perm} in Figure \ref{algos_update}
exhausts all these $3$-tuples.

For example, if $\bsb{s}$ and $\bsb{t}$ are the two sequences in 
Example \ref{un_example} with they difference  $(1,-2,1)$, $f=2$ and 
$x=0$, then the calls
\medskip
 
 {\tt transp($1,f,x$)}; \\
 \indent {\tt transp($-1,f+1,x+s[f]$)};\\

\noindent
transform $\bsb{s}$ into $\bsb{t}$ and 
$\sigma$ into  $\tau$. 

Algorithm {\tt Gen2\_Gray} provides $p$,
the rightmost position where the current sequence $\bsb{c}$ differs
from the previous generated one, and $u=\sum_{i=1}^p c_i$.
Algorithm {\tt Update\_Perm} uses $f$, the leftmost position where
$\bsb{c}$ differs from the previous generated sequence, and  
$x=\sum_{i=1}^{f-1}c_i$.

\begin{figure}[h]
\begin{center}
{\tt
\begin{tabular}{|c|}
\hline
\begin{minipage}[c]{.46\linewidth}
\begin{tabbing}
\hspace{0.75cm}\=\hspace{0.5cm}\=\hspace{1.8cm}\=\hspace{1.6cm}\=\hspace{1.9cm}
            \=\hspace{1.5cm}\kill
procedure Update\_Perm($p,u$)\\
\> $x:=u-c[p]-c[p-1]$;\\
\> if $p-2\geq 1$ and $s[p-2]=c[p-2]$\\
\> then $f:=p-1$;\\
\> else  $f:=p-2$; $x:=x-c[f]$;\\
\> endif\\
\> $(a_1,a_2):=(s[f]-c[f],s[f+1]-c[f+1])$;\\
\> if $f+2>n$ then $a_3:=0$; else $a_3:=s[f+2]-c[f+2]$; endif\\
\> if $a_1>0$ then $v:=1$; else $v:=-1$; endif\\
\> case $(a_1,a_2,a_3)$ of\\
\>\> $\pm(1,-1,0)$\> : transp($v,f,x$); \\
\>\> $\pm(2,-2,0)$\> : transp($v,f,x$); transp($v,f,x)$\\
\>\> $\pm(1,-2,1)$\> : transp($v,f,x$); transp($-v,f+1,x+s[f]$);\\
\>\> $\pm(1,-3,2)$\> : transp($v,f,x$); transp($-v,f+1,x+s[f]$);transp($-v,f+1,x+s(f]$);\\
\>\> $\pm(1,1,-2)$\> : transp($v,f+1,x+s[f]$); transp($v,f,x$); transp($v,f+1,x+s[f]$);\\
\>\> \, $(1,0,-1)$   \> : transp($1,f,x$); transp($1,f+1,x+s[f]$);\\
\>\> \, $(-1,0,1)$   \> : transp($-1,f+1,x+s[f]$); transp($-1,f,x$);\\
\> end case \\
end procedure.
\end{tabbing}
\end{minipage}\\
\hline
\end{tabular}
}
\caption{\label{algos_update} Algorithm {\tt Update\_Perm}.}
\end{center}
\end{figure}

Now, we sketch the generating algorithm for the set 
of permutations in $\SS_n$ having index $k$.
\begin{itemize}
\item initialize $\bsb{s}$ by the smallest, in co-lex order, sequence
      in $S(k,n)$ and $\sigma$ by the permutation in $\SS_n$ 
      having its McMahon code $\bsb{s}$,
\item run  {\tt Gen2\_Gray($k,n,0,0,0$)} where 
      {\tt output($p,u$)} is replaced by {\tt Update\_Perm($p,u$)}.
\end{itemize}

The obtained list of permutations is the image of the Gray code
$\mathcal{S}(k,n)$ through the bijection $\psi$ defined in relation 
(\ref{def_psi}); it consists of all permutations in $\SS_n$ with 
major index equal to $k$, and two consecutive permutations differ 
by at most three transpositions.
See Table \ref{list_pref} for the list of permutations in $\SS_6$
and with major index $4$.

\section{Final remarks}
\label{Conc}

Numerical evidences show that if we change the generating order 
of algorithm {\tt Gen\_Colex} as for 
{\tt Gen1\_Gray}, but without restricting it to subexcedant
sequences, then the obtained list for bounded compositions 
is still a Gray code with the closeness definition slightly relaxed:
two consecutive compositions differ in at most four adjacent positions.
Also, T. Walsh gives in \cite{Walsh} an efficient 
generating algorithm for a Gray code 
for bounded compositions of an
integer, and in particular for subexcedant sequences. 
In this Gray code two consecutive sequences
differ in two positions and by $1$ and $-1$ in these positions;
but these positions can be arbitrarily far, 
and so the image of this Gray code through the bijection 
$\psi$ defined by relation (\ref{def_psi})  in Section \ref{sec_Mc_code}
does not give a Gray code for
permutations with a fixed index.

\end{document}